\documentclass{article}

\usepackage{pstricks}

\usepackage{amsmath,amssymb}
\usepackage{amsfonts}
\usepackage{latexsym}
\usepackage{amsthm}

\begin{document}


\newtheorem{thm}{Theorem}[section]
\newtheorem{lem}[thm]{Lemma}
\newtheorem{conj}[thm]{Conjecture}
\newtheorem{prop}[thm]{Proposition}
\newtheorem{cor}[thm]{Corollary}
\newtheorem{sch}[thm]{Scholium}
\newtheorem{guess}[thm]{Guess}
\newtheorem{ex}[thm]{Example}
\theoremstyle{remark}
    \newtheorem*{rem}{{\bf Remark}}   
    \newtheorem*{rems}{{\bf Remarks}}

    \newtheorem*{note}{Note}
\theoremstyle{definition}
    \newtheorem{defn}[thm]{Definition}

\def\ds{\displaystyle}\def\C{{\mathcal C}}
\def\Z {{\mathbb Z}}
\def\AM {{\mathcal A}}
\def\BM {{\mathcal S}}
\def\CM {{\mathcal T}}
\def\K{{\mathcal K}}

\begin{center}
\vskip 1cm{\LARGE\bf {Motzkin numbers of higher rank: Generating
function and explicit expression}} \vskip 1cm
Toufik Mansour, Department of Mathematics, University of Haifa\\
31905 Haifa, Israel\\[4pt]

Matthias Schork\footnote{All correspondence should
be directed to this author.}, Alexanderstrasse 76\\
 60489 Frankfurt, Germany\\[4pt]

Yidong Sun, Department of Mathematics, Dalian Maritime University\\
116026 Dalian, P.R. China\\[4pt]

{\tt toufik@math.haifa.ac.il, mschork@member.ams.org,
sydmath@yahoo.com.cn}
\end{center}

\begin{abstract}
The generating function for the (colored) Motzkin numbers of higher
rank introduced recently is discussed. Considering the special case
of rank one yields the corresponding results for the conventional
colored Motzkin numbers for which in addition a recursion relation
is given. Some explicit expressions are given for the higher rank
case in the first few instances.
\end{abstract}

\vskip .2in

\renewcommand{\theequation}{\arabic{equation}}

\renewcommand{\theequation}{\arabic{equation}}

\section{Introduction}\label{Intro}

The classical Motzkin numbers (A001006 in \cite{OEIS}) count the
numbers of Motzkin paths (and are also related to many other
combinatorial objects, see Stanley \cite{stan}). Let us recall the
definition of Motzkin paths. We consider in the Cartesian plane
$\mathbb{Z}\times \mathbb{Z}$ those lattice paths starting from
$(0,0)$ that use the steps $\{U,L,D\}$, where $U=(1,1)$ is an
up-step, $L=(1,0)$ a level-step and $D=(1,-1)$ a down-step. Let
$M(n,k)$ denote the set of paths beginning in $(0,0)$ and ending in
$(n,k)$ that never go below the $x$-axis. Paths in $M(n,0)$ are
called {\it Motzkin paths} and $m_n:=|M(n,0)|$ is called $n$-{\it th
Motzkin  number}. Sulanke showed \cite{sul} that the Motzkin numbers
satisfy the recursion relation
\begin{equation}\label{mot}
(n+2)m_n=(2n+1)m_{n-1}+3(n-1)m_{n-2}
\end{equation}
and it is a classical fact (see Stanley \cite{stan}) that their
generating function is given by
\begin{equation}\label{genfunc}
\sum_{n\geq 0}m_nx^n=\frac{1-x-\sqrt{1-2x-3x^2}}{2x^2}.
\end{equation}
Those Motzkin paths which have no level-steps are called {\it Dyck
paths} and are enumerated by Catalan numbers (A000108 in
\cite{OEIS}), see Stanley \cite{stan}. In recent times the above
situation has been generalized by introducing colorings of the
paths. For example, the $k$-{\it colored Motzkin paths} have
horizontal steps colored by $k$ colors (see \cite{satsi,satsi2} and
the references given therein). More generally, Woan introduced
\cite{woan,woan2} colors for each type of step. Let us denote by $u$
the number of colors for an up-step $U$, by $l$ the number of colors
for a level-step $L$ and by $d$ the number of colors for a down-step
$D$. (Note that if we normalize the weights as $u+l+d=1$ we can view
the paths as discrete random walks.) One can then introduce the set
$M^{(u,l,d)}(n,0)$ of $(u,l,d)$-{\it colored Motzkin paths} and the
corresponding $(u,l,d)${\it -Motzkin numbers}
$m_n^{(u,l,d)}:=|M^{(u,l,d)}(n,0)|$. Woan has given \cite{woan} a
combinatorial proof that the $(1,l,d)$-Motzkin numbers satisfy the
recursion relation
\begin{equation}\label{rek}
(n+2)m_n^{(1,l,d)}=l(2n+1)m^{(1,l,d)}_{n-1}+(4d-l^2)(n-1)m^{(1,l,d)}_{n-2}.
\end{equation}
Choosing $l=1$ and $d=1$ yields the recursion relation (\ref{mot})
of the conventional Motzkin numbers $m_n\equiv m_n^{(1,1,1)}$. Note
that choosing $(u,l,d)=(1,k,1)$ corresponds to the $k$-colored
Motzkin paths. Defining $m_{k,n}:=|M^{(1,k,1)}(n,0)|$, one obtains
from (\ref{rek}) the recursion relation
$(n+2)m_{k,n}=k(2n+1)m_{k,n-1}+(4-k^2)(n-1)m_{k,n-2}$ for the number
of $k$-colored Motzkin paths. Sapounakis and Tsikouras derived
\cite{satsi} the following generating function for $m_{k,n}$:
\begin{equation}\label{kgenfunc}
\sum_{n\geq 0}m_{k,n}x^n=\frac{1-kx-\sqrt{(1-kx)^2-4x^2}}{2x^2}.
\end{equation}
For $k=1$ this identity reduces to (\ref{genfunc}) for the
conventional Motzkin numbers $m_n\equiv m_{1,n}$. One of the present
authors suggested \cite{ich3} (as ``Problem 1'') that it would be
interesting to find the recursion relation and generating function
for the general $(u,l,d)$-Motzkin numbers $m_n^{(u,l,d)}$. We will
prove in Theorem \ref{rankone} thet $m_n^{(u,l,d)}=m_n^{(1,l,ud)}$,
yielding the desired recursion relation. Furthermore, a generating
function and an explicit expression is derived for $m_n^{(u,l,d)}$.
In \cite{ich3} it was furthermore suggested to find the recursion
relation and the generating function for the (colored) Motzkin
numbers of higher rank (``Problem 2''). These numbers were
introduced by Schork \cite{ich} in the context of ``duality triads
of higher rank'' (and have also been considered before as
``excursions'', see, e.g., \cite{Band}). In view of this connection
Schork conjectured \cite{ich3} that the Motzkin numbers of rank $r$
satisfy a recursion relation of order $2r+1$. Very recently,
Prodinger was the first to observe \cite{prod} that this conjecture
does not hold by discussing explicitly the case $r=2$. In
particular, already for the first nontrivial cases $r=2,3$ the
relations involved become very cumbersome. We will describe in
Theorem \ref{thmm1} the generating function for the Motzkin numbers
of higher rank and discuss then several particular cases explicitly.

\section{Recursion relation and generating function for the general (colored) Motzkin numbers}

\begin{thm}\label{rankone}
The general $(u,l,d)$-Motzkin numbers satisfy the recursion relation
\begin{equation}\label{recone}
(n+2)m_n^{(u,l,d)}=l(2n+1)m^{(u,l,d)}_{n-1}+(4ud-l^2)(n-1)m^{(u,l,d)}_{n-2}.
\end{equation}
A generating function is given by
\begin{equation}\label{genone}
\sum_{n\geq0}m_n^{(u,l,d)}x^n=\frac{1-lx-\sqrt{(1-lx)^2-4udx^2}}{2udx^2},
\end{equation}
implying the explicit expression
\begin{equation}\label{exone}
m_n^{(u,l,d)}=\sum_{j=0}^{\frac{n}{2}}\frac{1}{j+1}\binom{2j}{j}\binom{n}{2j}u^jd^jl^{n-2j}.
\end{equation}
\end{thm}
\begin{proof}
Let us prove that $m_n^{(u,l,d)}=m_n^{(1,l,ud)}$ for all $n\geq0$.
In order to see that, let $U$ be any up-step in a Motzkin path; we
call $(U,D)$ a {\em pair} if the down-step $D$ is the first
down-step on the right-hand side of $U$ which has the same height as
$U$. The set of pairs of a Motzkin path is uniquely determined and
for each pair exist $ud$ possible combinations of colorings. The
same number of possible colorings result if the up-steps $u$ are
always colored white (i.e., $u=1$) and if each down-step can be
colored by $ud$ colors (which we call ``alternative colors'' to
distinguish them from the original colors). Thus, given a
$(u,l,d)$-Motzkin path we may replace the colors $(u_j,d_j)$ (with
$1\leq u_j\leq u$ and $1\leq d_j\leq d$) for the $j$-th pair $(U,D)$
by the combination of alternative colors $(1,c_j$) where $c_j$ is
the $([u_j-1]d+d_j)$-th alternative color (with $1\leq c_j \leq
ud$). Replacing the colors of all pairs in this fashion by the
alternative colors yields a $(1,l,ud)$-Motzkin path. Thus, we have
constructed a bijection between the set of $(u,l,d)$-Motzkin paths
of length $n$ and set of $(1,l,ud)$-Motzkin paths of length $n$,
thereby showing that $m_n^{(u,l,d)}=m_n^{(1,l,ud)}$. From this
identity and (\ref{rek}) we immediately obtain (\ref{recone}). An
equation for the generating function
$M_{(u,l,d)}(x):=\sum_{n\geq0}m_n^{(u,l,d)}x^n$ is obtained from the
``first return decomposition'' of a nonempty $(u,l,d)$-Motzkin path
$M$: either $M= LM'$ or $M=UM'DM''$, where $M',M''$ are
$(u,l,d)$-Motzkin paths. The two possibilities give the
contributions $lxM_{(u,l,d)}(x)$ and  $udx^2(M_{(u,l,d)}(x))^2$.
Hence, $M_{(u,l,d)}(x)$ satisfies
\[M_{(u,l,d)}(x)=1+lxM_{(u,l,d)}(x)+udx^2(M_{(u,l,d)}(x))^2,\]
yielding
\[ M_{(u,l,d)}(x)=\frac{1-lx-\sqrt{(1-lx)^2-4udx^2}}{2udx^2}\]
which is the asserted equation (\ref{genone}) for the generating
function. Note that this may also be written as
\[M_{(u,l,d)}(x)=\frac{1}{1-lx}C\left(\frac{udx^2}{(1-lx)^2}\right),\]
where
\[C(y)=\frac{1-\sqrt{1-4y}}{2y}=\sum_{n\geq0}\frac{1}{n+1}\binom{2n}{n}y^n\]
is the generating function for the Catalan numbers (see Stanley
\cite{stan}). Thus,
\[M_{(u,l,d)}(x)=\sum_{j\geq0}\frac{1}{j+1}\binom{2j}{j}\frac{u^jd^jx^{2j}}{(1-lx)^{2j+1}}.\]
Recalling $M_{(u,l,d)}(x)=\sum_{n}m_n^{(u,l,d)}x^n$, a comparison of
coefficients shows that the number of $(u,l,d)$-Motzkin paths of
length $n$ is given by (\ref{exone}).
\end{proof}

\section{Generating function for the general (colored) Motzkin numbers of higher rank}
We will now generalize the situation considered in the previous
section to the case of higher rank. One of the present authors
discussed \cite{ich} in the context of duality triads of higher rank
(where one considers recursion relations of higher order, or
equivalently, orthogonal matrix polynomials \cite{ich2}) why it is
interesting to consider the situation where the steps of the paths
can go up or down more than one unit. The maximum number of units
which a single step can go up or down will be called the rank. More
precisely, let $r\geq 1$ be a natural number. The set of {\it
admissable} steps consists of:
\begin{enumerate}
\item $r$ types of up-steps $U_j=(1,j)$ with weights $u_j$ for $1\leq j \leq r$.
\item A level-step $L=(1,0)$ with weight $l$.
\item $r$ types of down-steps $D_j=(1,-j)$ with weights $d_j$ for $1\leq j \leq r$.
\end{enumerate}
In the following we write $({\bf u},l,{\bf
d}):=(u_r,\ldots,u_1,l,d_1,\ldots,d_r)$ for the vector of weights.
\begin{defn}\cite{ich}
The set $M^{({\bf u},l,{\bf d})}(n,0)$ of $({\bf u},l,{\bf d})${\it
-colored Motzkin paths of rank r and length $n$} is the set of paths
which start in $(0,0)$, end in $(n,0)$, have only admissable steps
and are never below the $x$-axis. The corresponding number of paths,
$m_n^{({\bf u},l,{\bf d})}:=|M^{({\bf u},l,{\bf d})}(n,0)|$, will be
called $({\bf u},l,{\bf d})$-{\it Motzkin number of rank $r$}.
\end{defn}

Motzking paths of higher rank were considered in the literature
already before \cite{ich} under the name ``excursions'', see, e.g.,
\cite{Band}. In \cite{ich3} it is discussed how one may associate
with each Motzkin path of rank $r$ and length $n$ a conventional
Motzkin path of length $rn$ in a straightforward fashion. However,
it was also discussed that this association is no bijection and that
the case of higher rank is more subtle. In the following we derive
an equation that the generating function for the Motzkin numbers of
higher rank satisfies. In order to do that we need the following
notations. We denote the set of all $({\bf u},l,{\bf d})$-Motzkin
paths of rank $r$ that start at height $s$ and end at height $t$
(and have only admissable steps and are never below the $x$-axis) by
$\AM_{s,t}$, and we denote the subset of paths of length $n$ by
$\AM_{s,t}(n)$. Define
$$A_{s,t}\equiv A_{s,t}(x):=\sum_{n\geq0}|\AM_{s,t}(n)|x^n.$$
We extend this notation by defining $A_{s,t}=0$ for all $s<0$ or
$t<0$. For $s,t\geq 0$ we denote the subset of paths in $\AM_{s,t}$
that never touch the $x$-axis by $\AM_{s,t}^{\ast}$. In the
following we make several times use of the relation
$\AM_{s,t}^{\ast}\simeq \AM_{s-1,t-1}$, or,
$|\AM_{s,t}^{\ast}(n)|=|\AM_{s-1,t-1}(n)|$, implying
\begin{equation}\label{irr}
A_{s,t}^{\ast}(x)
:=\sum_{n\geq0}|\AM_{s,t}^{\ast}(n)|x^n=A_{s-1,t-1}(x).
\end{equation}

\begin{thm}\label{thmm1}
The generating function $A_{0,0}\equiv M_{({\bf u},l,{\bf
d})}(x)=\sum_{n\geq 0}m_n^{({\bf u},l,{\bf d})}x^n$ for the general
$({\bf u},l,{\bf d})$-Motzkin numbers of rank $r$ satisfies
\begin{equation}\label{genfunct}
A_{0,0}=1+lxA_{0,0}+x^2A_{0,0}\sum_{p=1}^r\sum_{q=1}^ru_pd_qA_{p-1,q-1},
\end{equation}
where the generating functions $A_{i,j}$ with $1\leq i\leq r-1$ and
$0\leq j\leq r-1$ satisfy
$$A_{i,j}=A_{i-1,j-1}+xA_{0,j}\sum_{q=1}^rd_qA_{i-1,q-1},$$
and for all $1\leq j\leq r-1$,
$$A_{0,j}=xA_{0,0}\sum_{p=1}^ru_pA_{p-1,j-1}.$$
\end{thm}
\begin{proof}
From the definitions, a nonempty $({\bf u},l,{\bf d})$-Motzkin path
$M$ can start either by a level step $L$ or an up-step $U_p$ with
$1\leq p \leq r$.
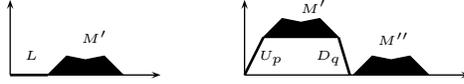
\begin{figure}[ht]
\begin{center}
\begin{pspicture}(3,1)
\psline[linewidth=.5pt]{->}(0,0)(2,0)\psline[linewidth=.5pt]{->}(0,0)(0,1)
\psline[linewidth=1pt](0,0)(.5,0)
\put(0.25,-.25){\psline[linewidth=.5pt,fillstyle=solid,fillcolor=black](1.25,.25)(.25,.25)(.5,.5)(.75,.45)(1,.5)(1.25,.25)}
\put(0.2,0.2){\tiny $L$}\put(0.95,0.43){\tiny $M'$}
\end{pspicture}
\begin{pspicture}(3,1)
\psline[linewidth=.5pt]{->}(0,0)(3,0)\psline[linewidth=.5pt]{->}(0,0)(0,1)
\psline[linewidth=1pt](0,0)(.25,.5)\psline[linewidth=1pt](1.25,.5)(1.4,0)
\put(0,.25){\psline[linewidth=.5pt,fillstyle=solid,fillcolor=black](1.25,.25)(.25,.25)(.5,.5)(.75,.45)(1,.5)(1.25,.25)}
\put(1.2,-.25){\psline[linewidth=.5pt,fillstyle=solid,fillcolor=black](1.25,.25)(.25,.25)(.5,.5)(.75,.45)(1,.5)(1.25,.25)}
\put(0.2,0.2){\tiny $U_p$}\put(.95,0.2){\tiny $D_q$}
\put(0.75,0.83){\tiny $M'$}\put(1.77,0.4){\tiny $M''$}
\end{pspicture}
\caption{First return decomposition of a Motzkin path of rank
$r$.}\label{ffig0}
\end{center}
\end{figure}
In the case the path starts by an up-step $U_p$, we use the ``first
return decomposition`'' of $M$, i.e., we write $M$ as
$M=U_pM'D_qM''$ where $M''$ is an arbitrary $({\bf u},l,{\bf
d})$-Motzkin path and $M'\in\AM_{p,q}$ such that $M'$ does not touch
the height zero, see Figure~\ref{ffig0}. Thus, $M' \in
\AM_{p,q}^{\ast}$ and the generating function $A_{0,0}$ satisfies
\[A_{0,0}=1+lxA_{0,0}+x^2A_{0,0}\sum_{p=1}^r\sum_{q=1}^ru_pd_qA_{p-1,q-1}
\]
where we have used (\ref{irr}). This shows (\ref{genfunct}). Now,
let us write an equation for the generating function $A_{i,j}$ (with
$1\leq i\leq r-1$ and $0\leq j\leq r-1$) for the number of $({\bf
u},l,{\bf d})$-Motzkin paths $Q_{ij}=U_iP_{ij}D_j$ where $P_{ij}\in
\AM_{i,j}$. In order to do that, we use the first return
decomposition of such paths:
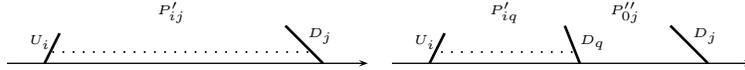
\begin{figure}[ht]
\begin{center}
\begin{pspicture}(5,1)
\psline[linewidth=.5pt]{->}(-.5,0)(4.3,0)
\psline[linewidth=1pt](0,0)(.2,.4)
\put(-1.5,0){\psline[linewidth=1pt](4.7,.5)(5.2,0)}
\put(-.2,0.25){\tiny$U_i$}\put(3.5,.35){\tiny$D_{j}$}\put(1.5,.65){\tiny$P'_{ij}$}
\psline[linestyle=dotted](.1,.15)(3.5,.15)
\end{pspicture}
\begin{pspicture}(5,1)
\psline[linewidth=.5pt]{->}(-.5,0)(4.3,0)
\psline[linewidth=1pt](0,0)(.2,.4)
\put(-1.5,0){\psline[linewidth=1pt](4.7,.5)(5.2,0)}
\put(2,0){\psline[linewidth=1pt](0,0)(-.2,.5)}
\put(-.2,0.25){\tiny$U_i$}\put(3.5,.35){\tiny$D_{j}$}\put(2,.25){\tiny$D_{q}$}
\psline[linestyle=dotted](.15,.15)(1.9,.15)\put(.8,.65){\tiny$P'_{iq}$}\put(2.4,.65){\tiny$P''_{0j}$}
\end{pspicture}
\caption{The case $1\leq i\leq r-1$ and $0\leq j\leq
r-1$.}\label{ffig1}
\end{center}
\end{figure}
either $Q_{ij}=U_iP'_{ij}D_j$ such that $P'_{ij}$ does not touch the
height zero (i.e., $P'_{ij}\in \AM_{i,j}^{\ast}$), or
$Q_{ij}=U_iP'_{iq}D_qP''_{0j}D_j$ such that $D_q$ is the first
down-step that touches the height zero (thus, $P'_{iq}$ does not
touch the height zero, i.e., $P'_{iq}\in \AM_{i,q}^{\ast}$) as
described in Figure~\ref{ffig1}. Using (\ref{irr}), the generating
function $A_{i,j}$ thus satisfies
             $$A_{i,j}=A_{i-1,j-1}+xA_{0,j}\sum_{q=1}^rd_qA_{i-1,q-1},$$
as claimed.

In order to write an equation for the generating function $A_{0,j}$
(with $0\leq j\leq r-1$), let us consider the last up-step $U_p$
from height zero (there must exist at least one such up-step). In
this case each $({\bf u},l,{\bf d})$-Motzkin
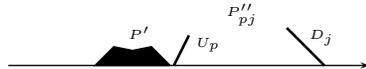
\begin{figure}[ht]
\begin{center}
\begin{pspicture}(5,1)
\psline[linewidth=.5pt]{->}(-.5,0)(4.3,0)
\put(-1.5,0){\psline[linewidth=1pt](4.7,.5)(5.2,0)}
\put(1.7,0){\psline[linewidth=1pt](0,0)(.2,.4)}
\put(3.5,.35){\tiny$D_{j}$}\put(2,.25){\tiny$U_{p}$}
\put(2.4,.65){\tiny$P''_{pj}$}\put(1.1,.35){\tiny$P'$}
\put(0.4,-.25){\psline[linewidth=.5pt,fillstyle=solid,fillcolor=black](1.25,.25)(.25,.25)(.5,.5)(.75,.45)(1,.5)(1.25,.25)}
\end{pspicture}
\caption{The case $i=0$ and $1\leq j\leq r-1$.}\label{ffig2}
\end{center}
\end{figure}
path can be decomposed as $P'U_pP''_{pj}D_j$, where $P'$ is an
arbitrary $({\bf u},l,{\bf d})$-Motzkin path and
$P''_{pj}\in\AM_{p,j}$ such that it does not touch the height zero
(i.e., $P''_{pj}\in\AM_{p,j}^{\ast}$), as described in
Figure~\ref{ffig2}. Thus, using again (\ref{irr}), the generating
function $A_{0,j}$ satisfies
$$A_{0,j}=xA_{0,0}\sum_{p=1}^ru_pA_{p-1,j-1},$$
which completes the proof.
\end{proof}

\begin{ex}
Let us consider as an example the case $r=1$. It follows from
Theorem~\ref{thmm1} that
\begin{equation}\label{rone}
0=1+(lx-1)A_{0,0}+u_1d_1x^2A_{0,0}^2,
\end{equation}
yielding
$$A_{0,0}\equiv M_{(u_1,l,d_1)}(x)=\frac{1-lx-\sqrt{(1-lx)^2-4u_1d_1x^2}}{2u_1d_1x^2}$$
as described in the proof of Theorem \ref{rankone}. This implies -
according to \eqref{exone} - that
$$m_n^{(u_1,l,d_1)}=\sum_{j=0}^{n/2}\frac{1}{1+j}\binom{2j}{j}\binom{n}{2j}l^{n-2j}(u_1d_1)^{j}.$$
\end{ex}

\begin{ex}\label{exr1}
As another example, it follows from Theorem \ref{thmm1} that one has
for $r=2$ that
$$\left\{\begin{array}{rl}
A_{0,0}&=1+lxA_{0,0}+x^2A_{0,0}(u_1d_1A_{0,0}+u_1d_2A_{0,1}+u_2d_1A_{1,0}+u_2d_2A_{1,1}),\\
A_{0,1}&=xA_{0,0}(u_1A_{0,0}+u_2A_{1,0}),\\
A_{1,0}&=xA_{0,0}(d_1A_{0,0}+d_2A_{0,1}),\\
A_{1,1}&=A_{0,0}+xA_{0,1}(d_1A_{0,0}+d_2A_{0,1}).\\
\end{array}\right.$$
Solving the above system of equations we obtain that
\begin{eqnarray*}
0&=& 1+(lx-1)A_{0,0}-x^2(d_2u_2-d_1u_1)A_{0,0}^2+x^2(xu_1^2d_2+xu_2d_1^2-2xlu_2d_2+2d_2u_2)A_{0,0}^3\\
&&-u_2d_2x^4(d_2u_2-d_1u_1)A_{0,0}^4+x^4u_2^2d_2^2(lx-1)A_{0,0}^5+x^6u_2^3d_2^3A_{0,0}^6.
\end{eqnarray*}
Note that, if we set $u_2=d_2=0$ in the above expression then we get
\eqref{rone}. It is also interesting to consider the case $({\bf
u},l,{\bf d})=({\bf 1},1,{\bf 1})$ of non-colored Motzkin paths of
rank 2, representing the most natural generalization of the
conventional Motzkin paths. The above equation reduces in this case
to
\begin{eqnarray*}
0&=&1+(x-1)A_{0,0}
+2x^2A_{0,0}^3+x^4(x-1)A_{0,0}^5+x^6A_{0,0}^6\\
&=&(1+xA_{0,0})^2(1-(x+1)A_{0,0}+x(x+2)A_{0,0}^2-x^2(x+1)A_{0,0}^3
+x^4A_{0,0}^4).
\end{eqnarray*}
Setting $A_{0,0}\equiv M_{({\bf 1},1,{\bf 1})}(x)$ and using that
$A_{0,0}$ is a formal power series (thus,
$A_{0,0}\neq-\frac{1}{x}$), this is equivalent to
$$0=1-(x+1)M_{({\bf 1},1,{\bf 1})}(x)+x(x+2)M_{({\bf 1},1,{\bf 1})}^2(x)-x^2(x+1)M_{({\bf 1},1,{\bf 1})}^3(x)
+x^4M_{({\bf 1},1,{\bf 1})}^4(x),$$ as described already in
\cite{Band,prod}. To obtain a recursion relation for $m_n\equiv
m_n^{({\bf 1},1,{\bf 1})}$, one can use the MAPLE program package
{\rm gfun} written by Salvy et al. \cite{savy}. Prodinger has done
this \cite{prod} to obtain {\small$$\begin{array}{l} 625(n + 3)(n +
2)(n + 1)m_n - 125(n + 3)(n + 2)(7n + 27)m_{n+1}-50(n + 3)(5n^2 +
24n + 23)m_{n+2}\\ + (41890 + 30860n + 7540n^2 +
610n^3)m_{n+3}-(6844 + 5151n + 1214n^2 + 91n^3)m_{n+4}\\ \qquad- (n
+ 7)(23n^2 + 301n + 976)m_{n+5} + 2(2n + 13)(n + 8)(n + 7)m_{n+6} =
0
\end{array}$$}
and mentions that Salvy has informed him that this recursion of
order $6$ is minimal. Thus, the Motzkin numbers of rank $r=2$
satisfy a $7$-term recursion relation and not a $2\cdot2+1=5$-term
relation as the conjecture of Schork \cite{ich3} implies. Thus, the
conjecture does not hold! This was first observed by Prodinger
\cite{prod}. The first few values of $m_n\equiv m_n^{({\bf 1},1,{\bf
1})}$ are given in Table \ref{T1}; this sequence is sequence A104184
in \cite{OEIS}.

\renewcommand{\arraystretch}{1.5}
\begin{table}[ht]
\begin{center}
\begin{tabular}{|r|r|r|r|r|r|r|r|r|r|r|r|}\hline
      & $n=1$&2 & 3&4&5&6&7&8&9&10\\ \hline \hline
    $m_n^{({\bf 1},1,{\bf 1})}$ & 1  & 3 & 9& 32 & 120& 473& 1925 & 8034&34188&147787 \\ \hline
     \end{tabular}
\caption{The first few values of the non-colored Motzkin numbers of
rank two.}\label{T1}
\end{center}
\end{table}
\end{ex}

The above example shows that the general case (where the weights are
not restricted to $1$) seems to be extremely complicated (see
Theorem~\ref{thmm1}). Thus, from now on let us consider the case
where all weights are equal to $1$, i.e., $u_i=l=d_i=1$ for $1\leq i
\leq r$. In this case we denote the set of paths by ${\cal B}_{s,t}$
or ${\cal B}_{s,t}(n)$ and the generating function by $B_{i,j}$
(instead of $A_{i,j}$). Then one has
\begin{equation}\label{smm}
B_{s,t}=B_{t,s}
\end{equation}
which can be easily seen since to each path $ P\in {\cal
B}_{s,t}(n)$ one can associate a path $P' \in {\cal B}_{t,s}(n)$ by
traversing $P$ in opposite direction. Since this is clearly a
bijection the above equation follows. Theorem \ref{thmm1} shows that
one has to solve in the general case a system of $r^2$ equations for
the $r^2$ unknowns $A_{i,j}$. Due to the symmetry (\ref{smm}) this
reduces in the case where all weights are equal to 1 to a system of
$\frac{r(r+1)}{2}$ equations in the $\frac{r(r+1)}{2}$ unknowns
$B_{i,j}$ (where $r-1\geq i\geq j\geq 0$).

\begin{ex} Theorem \ref{thmm1} with \eqref{smm} gives for $r=3$ the following set of $\frac{3\cdot 4}{2}=6$ equations
$$\left\{\begin{array}{rl}
B_{0,0}&=1+xB_{0,0}+x^2B_{0,0}(B_{0,0}+2B_{1,0}+2B_{2,0}+2B_{2,1}+B_{1,1}+B_{2,2}),\\
B_{1,0}&=xB_{0,0}(B_{0,0}+B_{1,0}+B_{2,0}),\\
B_{2,0}&=xB_{0,0}(B_{1,0}+B_{1,1}+B_{2,1}),\\
B_{1,1}&=B_{0,0}+xB_{1,0}(B_{0,0}+B_{1,0}+B_{2,0}),\\
B_{2,1}&=B_{1,0}+xB_{1,0}(B_{1,0}+B_{1,1}+B_{2,1}),\\
B_{2,2}&=B_{1,1}+xB_{2,0}(B_{1,0}+B_{1,1}+B_{2,1}).
\end{array}\right.$$
Solving the above system of equations we obtain that the generating
function $B_{0,0}\equiv M_{({\bf 1},1,{\bf 1})}(x)$ for the
non-colored Motzkin numbers of rank 3 satisfies
$$0=1-(1+x)B_{0,0}+2xB_{0,0}^2+x^2(1-2x)B_{0,0}^4+2x^5B_{0,0}^6-x^6(1+x)B_{0,0}^7+x^8B_{0,0}^8.$$
Now, we would like to give a recursion relation for the sequence
$m_n^{({\bf 1},1,{\bf 1})}$. This can be automatically done with
MAPLE's program {\rm gfun}: The procedure ``{\rm algeqtodiffeq}''
translates the (algebraic) equation for $B_{0,0}\equiv M_{({\bf
1},1,{\bf 1})}(x)$ into an equivalent differential equation of order
$7$ (it is too long to present here) and then the procedure ``{\rm
diffeqtorec}'' translates the differential equation into a $28$-term
recursion relation. The first few values of $m_n^{({\bf 1},1,{\bf
1})}$ are given in Table \ref{T2}; this sequence seems not to be
listed in \cite{OEIS}.
\renewcommand{\arraystretch}{1.5}
\begin{table}[ht]
\begin{center}
\begin{tabular}{|r|r|r|r|r|r|r|r|r|r|r|r|}\hline
      & $n=1$&2 & 3&4&5&6&7&8&9&10\\ \hline \hline
    $m_n^{({\bf 1},1,{\bf 1})}$ & 1  & 4 & 16& 78 & 404 & 2208& 12492 &72589& 430569& 2596471\\ \hline
     \end{tabular}
\caption{The first few values of the non-colored Motzkin numbers of
rank three.}\label{T2}
\end{center}
\end{table}
\end{ex}

\begin{ex} Theorem \ref{thmm1} with \eqref{smm} gives for $r=4$ a set of $\frac{4\cdot 5}{2}=10$ equations for the $B_{i,j}$. As above, it is possible to solve this system and obtain that the generating
function $B_{0,0}\equiv M_{({\bf 1},1,{\bf 1})}(x)$ for the
non-colored Motzkin numbers of rank 4 satisfies the following
equation
\begin{eqnarray*}
0&=&1+(x-1)B_{0,0}-2xB_{0,0}^2-x(x+2)(x-1)B_{0,0}^3-x^2(x-2)(x+2)B_{0,0}^4\\
&&+x^2(x-1)B_{0,0}^5+x^3(x-2)(x+1)^2B_{0,0}^6 +x^4(x+1)(x-1)^2B_{0,0}^7-x^5(2x^2-3x-4)B_{0,0}^8\\
&&+x^6(x+1)(x-1)^2B_{0,0}^9+x^7(x-2)(x+1)^2B_{0,0}^{10}+x^8(x-1)B_{0,0}^{11}-x^{10}(x-2)(x+2)B_{0,0}^{12}\\
&&-x^{11}(x+2)(x-1)B_{0,0}^{13}-2x^{13}B_{0,0}^{14}+x^{14}(x-1)B_{0,0}^{15}+x^{16}B_{0,0}^{16}.
\end{eqnarray*}
The first few values of the corresponding Motzkin numbers
$m_n^{({\bf 1},1,{\bf 1})}$ of rank 4 are given in Table \ref{T3};
this sequence seems not to be listed in \cite{OEIS}.
\renewcommand{\arraystretch}{1.5} \begin{table}[ht] \begin{center}
\begin{tabular}{|r|r|r|r|r|r|r|r|r|r|r|r|}\hline
& $n=1$&2 & 3&4&5&6&7&8&9&10\\ \hline \hline $m_n^{({\bf 1},1,{\bf
1})}$ & 1 & 5 & 25& 155 & 1025 & 7167& 51945 &387000& 2944860&
22791189\\ \hline
\end{tabular} \caption{The first few values of the non-colored Motzkin numbers of
rank four.}\label{T3} \end{center}
\end{table}
\end{ex}

\begin{rem}
Let us consider the non-colored Motzkin numbers of rank $r$ (i.e.,
all weights are equal to 1). They satisfy a $\tau(r)$-term recursion
relation where $\tau(r)$ is defined by this property (and is
minimal). This yields a well-defined sequence $\{\tau(r)\}_{r \in
\mathbb{N}}$ starting - according to the above examples - with
$3,7,28,\ldots$. The original conjecture $\tau(r)=2r+1$ proved to be
too naive, but it seems that an exact formula for $\tau(r)$ will be
very difficult to obtain. Thus, approximations or bounds would also
be interesting. Considering the equations for the generating
functions $B_{0,0}\equiv M_{({\bf 1},1,{\bf 1})}(x)$ of non-colored
Motzkin numbers given in the above examples for rank $r\leq 4$, one
observes that they all have the form $f(x,B_{0,0})=0$ where $f(x,y)$
is a polynomial of degree $2^r$ in $y$ with coefficients in
$\mathbb{Z}[x]$ (and with constant coefficient 1); more precisely,
the coefficient $a_i(x)$ of $y^i$ is a polynomial of degree at most
$i$ in $x$ over the integers $\mathbb{Z}$. Thus, we can write
\begin{equation}
0=1+\sum_{i=1}^{2^r}\sum_{j=0}^{i}a_{i,j}x^j(B_{0,0})^i
\end{equation}
where $a_{i,j}\in \mathbb{Z}$. It would be interesting to find out
whether this representation holds for all ranks (which is what we
expect) or whether it is confined to $r\leq 4$.
\end{rem}

We would like to close this paper by stressing that it is still an
open problem to derive a recursion relation for the Motzkin numbers
of higher rank. As the explicit examples $r=2,3,4$ and the
complicated set of equations for the generating function given in
Theorem \ref{thmm1} show this will be a rather daunting task (see
also the last remark).

\section{Acknowledgments}
The authors would like to thank Helmut Prodinger for sending them
\cite{prod} prior to publication and drawing their attention to
\cite{Band}.

\bigskip
\hrule
\bigskip

\noindent 2000 {\it Mathematics Subject Classification}:
Primary 05A15, 11B37, 11B83.

\noindent \emph{Keywords: } Motzkin number; Catalan number;
recursion relation; generating function.

\end{document}